\documentclass[12pt]{article}

\usepackage{amsfonts}
\usepackage{amsmath}
\usepackage{amssymb,amsthm}
\usepackage{fixltx2e}
\usepackage[margin=1.2in]{geometry}
\usepackage[english]{babel}
\usepackage[utf8]{inputenc}
\usepackage[T1]{fontenc}

\usepackage{amsmath,amssymb}

\usepackage{graphicx}
\usepackage[colorlinks=true,citecolor=black,linkcolor=black,urlcolor=blue]{hyperref}

\usepackage{mathtools}
\usepackage{mathrsfs}
\def\Z{\mathbb{Z}}
\def\N{\mathbb{N}}
\mathtoolsset{showonlyrefs}

\MakeRobust{\overrightarrow}

\newtheorem{theorem}{Theorem}

\newtheorem{remark}[theorem]{Remark}

\begin{document}

\title{Bounds on the cardinality of restricted sumsets in $\mathbb{Z}_{p}$}

\author{Gabriel Bengochea\thanks{Supported by SEP--PRODEP through project UAM-PTC-630.}\qquad Bernardo Llano\\
\small Departamento de  Matemáticas\\
\small Universidad Autónoma Metropolitana, Iztapalapa (UAM-I)\\
\small Apartado 55-534, M\'exico D.F.\\
\small \tt \{g.bengochea, llano\}@xanum.uam.mx\\
}

\maketitle

\begin{abstract}

\noindent In this paper we present a procedure which allows to
transform a subset $A$ of $\mathbb{Z}_{p}$ into a set $ A'$ such
 that $ |2\hspace{0.15cm}\widehat{} A'|\leq|2\hspace{0.15cm}\widehat{} A | $, where $2\hspace{0.15cm}\widehat{} A$ is
 defined to be the set $\left\{a+b:a\neq b,\;a,b\in A\right\}$. From this result,
 we get some lower bounds for  $ |2\hspace{0.15cm}\widehat{} A| $. Finally, we give some remarks related to the
 problem for which sets $A\subset \mathbb{Z}_{p}$ we have the equality $|2\hspace{0.15cm}\widehat{} A|=2|A|-1$.
\vspace{0.5cm}
\noindent \emph{Keywords}: restrited sumsets;
addition theorems; Erd\H{o}s-Heilbronn problem

\noindent  \emph{Math. Subj. Class.}: 11P99, 11B2
\end{abstract}

\section{Introduction}
In Additive Number Theory, two types of problem are studied: the so
called direct problems and the inverse ones. Let  $G$ be an Abelian
group and a nonempty finite subset $A\subset G$. The problem of
determining properties of a sumset $$2A=A+A=\left\{ a+b:a,b\in
A\right\}$$ from properties of $A$ is called a \textit{direct
problem}. \textit{Inverse problems} consist of determining
properties of set $A$ from properties of sumset $2A$. These problems
are also studied for the so-called restricted sumsets
$$2\hspace{0.15cm}\widehat{}A = \left\{
a+b:a\neq b,\;a,b\in A\right\}$$

In 1963, the Erd\H{o}s-Heilbronn conjecture \cite{Erdos} was
established by Paul Erd\H{o}s in a conference at the University of
Colorado. The conjecture states that
\begin{equation}
\left\vert 2\hspace{0.15cm}\widehat{}A\right\vert  \geq \min
\{p,2k-3\},
\end{equation}%
where $A\subset \mathbb{Z}_{p}$ and $p$ is a prime number.

Several mathematicians obtained partial results of the conjecture,
for example Rickert \cite{Rick}, Mansfield \cite{Mans}, R\"{o}dseth
\cite{Rod} and Freiman, Low and Pitman \cite{Freiman}. After more
than 30 years, Dias da Silva and Hamidoune \cite{Dias} published a
proof of the conjecture. They showed the more general inequality
\begin{equation}
\left\vert h\hspace{0.15cm}\widehat{}A\right\vert =\left\vert
\{a_{1}+a_{2}+...+a_{h}:a_{i}\in A,\;a_{i}\neq a_{j},\;i\neq
j\}\right\vert \geq \min \{p,h\ |A|-h^{2}+1\},  \nonumber
\end{equation}
where $h\geq 2$. The proof is based on the representation theory of
symmetric groups and advance linear algebra. A bit after, Alon,
Nathanson and Ruzsa \cite{alon1995adding,alon1996polynomial}  gave a
shorter proof using the so-called polynomial method. In 2005,
K\'{a}rolyi \cite{Karolyi} proved that the equality in the
conjecture holds for $h=2$ when the set $A$ ($\left\vert
A\right\vert \geq 5$) is an arithmetic progression even for Abelian
groups. These are typical examples of inverse problems for
restricted sums. In particular, it is worth stating the main result
of the paper mentioned before.

\begin{theorem}\label{Kar}
Let $A\subset \mathbb{Z}_{p}$ such that $p$ is a prime number, $
p>2\left\vert A\right\vert -3$ and $\left\vert A\right\vert \geq 5.$
Then $\left\vert 2\hspace{0.15cm}\widehat{}A\right\vert =2\left\vert
A\right\vert-3$ if and only if $A$ is an arithmetic progression.
\end{theorem}

More recently, K\'{a}rolyi and Paulin \cite{MR3082208} present a
shorter proof of the result stated by Dias da Silva and Hamidoune
for the case $h=2$ using the same techniques. Moreover, they use
this method to obtain the above theorem.

We notice that as a consequence of Theorem \ref{Kar}, one can
conclude that $2\hspace{0.15cm}\widehat{}A$ is also an arithmetic
progression. In this paper, we go further and consider sets
$A\subset \mathbb{Z}_{p}$ not in arithmetic progression and ask for
bounds of $\left\vert 2\hspace{0.15cm}\widehat{}A\right\vert $. For
this purpose, let us define $A=[0,l-1]\cup B$ written in its normal
form, where $B\subset \mathbb{Z}_{p}$ is arbitrary set. In this
case, the restricted sumset $2\hspace{0.15cm}\widehat{}A$ can be
expressed as
\begin{equation}
2\hspace{0.15cm}\widehat{}A=2\;\widehat{}\;[0,l-1]\cup \left(
\lbrack 0,l-1]+B\right) \cup 2\hspace{0.15cm}\widehat{}\,B.
\nonumber
\end{equation}%
The main goal of this paper is to find lower bounds of
$|2\hspace{0.15cm} \widehat{}A|$ in the cases when
\begin{itemize}
\item[i)] $2\;\widehat{}\;[0,l-1]\cup \left( B+[0,l-1]\right) $ is an interval,

\item[ii)] $2\hspace{0.15cm}\widehat{}A$ is an interval and

\item[iii)] $2\;\widehat{}\;[0,l-1]\cup 2\hspace{0.15cm}\widehat{}\,B$ is an
interval.
\end{itemize}

The paper is organized as follows. Section \ref{S2} contains some
preliminary results. In Section \ref{S3}, we prove the already
mentioned lower bounds for $|2\hspace{0.15cm}\widehat{}A|$. As
consequences of the theorems established, we can conclude that in
Cases i) and ii), we have that $\left\vert
2\hspace{0.15cm}\widehat{}A\right\vert \geq 2\left\vert A\right\vert
-1.$ In Case iii), the same inequality is valid for some special
instances of Theorem \ref{t3}.

\section{Preliminaries} \label{S2}
We define $\Z_{p}$ as the additive group of congruence classes
modulo $p$, where $p$ is a prime number.  We define a total order
$\leq$ in $\Z_p$ as follows: for $\overline{a_1},
\overline{a_2}\in\Z_p$, $\overline{a_1}\leq \overline{a_2}\,$ if and
only if $a_1\leq a_2$ for $a_1,a_2\in\{0,1,2,...,p-1\}.$  We will
omit the over line of the elements of $\Z_p$ when there is no
confusion. We define the function $w:\Z_p\times\Z_p\rightarrow\N_0$
as
$$w({a_1},{a_2}) = \left\{
  \begin{array}{cl}
     \left |a_1-a_2\right |-1,& \qquad{a_1}\ne{a_2}\\
    0, & \qquad {{a_1}={a_2}}
  \end{array}. \right.$$
The function $w$ measure the width between two elements of $ \Z_p $.
For instance, $w({0},{p})=0$, $w({2},{3})=0$, $w({0},{p-1})=p-2$.
For $A\subseteq\Z_p$, let us define
\begin{equation}\notag
    \min(A):=\left\{a\in A:a\leq b,\, \forall\, b\in A\right\},
\end{equation}
\begin{equation}\notag
    \max(A):=\left\{a\in A:a\geq b,\, \forall\, b\in A\right\},
\end{equation}
and
\begin{equation}\notag
    \min(\emptyset)=\max(\emptyset)=\emptyset.
\end{equation}
 For $a_1,a_2\in\Z_p$, with $a_1\leq a_2$, we define the integer interval as
 \begin{equation}\notag
[a_1,a_2]=\left \{ \begin{array}{lr}
  \{a_1,a_1+1,a_1+2,...,a_2\},& a_1\neq a_2\\[7pt]
  \{a_1\},& a_1=a_2
 \end{array}\right ..
 \end{equation}
Let $A$ and $B$ be nonempty subsets of $\Z_p$. We define the
(unrestricted) sumset
\begin{equation}\notag
A+B=\left\{ x\mid x=a+b,a\in A,b\in B\right\},
\end{equation}
and the restricted sumset
\begin{equation}\notag
A\;\widehat{+}\;B=\left\{ x\mid x=a+b,a\in A,b\in B,a\neq b\right\}
.
\end{equation}
We enumerate some other definitions ($t\in\Z_p$):
\begin{equation}\notag
\begin{array}{rcl}
A-B &=&A+\left( -B\right) , \\
-A &=&\left\{ -a\mid a\in A\right\} , \\
t*A &=&\left\{ ta\mid a\in A\right\} , \\
A\setminus \{t\} &=&\{a\in A\mid a\neq t\}, \\
A-t &=&\{a-t\mid a\in A\},\quad t\in\Z_p , \\
\overline{A} &=&\{a^{\prime }\in \mathbb{Z}_{p}\mid a^{\prime
}\notin A\}.
\end{array}
\end{equation}
Let $a,d\in\Z_p$, with $d\neq 0$ and $r\in\N$. The set
\begin{equation}\label{ap}
\begin{split}
\{a+id:0\leq i\leq r-1\}
\end{split}
\end{equation}
is an arithmetic progression of length $r$.
\section{Results }\label{S3}
Let $A'$ be a subset of $\Z_p$ and $\left|A'\right|=m\geq 2$. Since
any set with two elements is an arithmetic progression, then there
always exists a maximum subset of $A'$ which is an arithmetic
progression with at least two elements. We suppose that the maximum
arithmetic progression of $A'$ has $l\geq 2$ elements. It follows
that $A'$ can be written as
\begin{equation}\notag
 \begin{split}
 A'=\{a+id:i=0,1,2,...,l-1\}\cup B',
 \end{split}
 \end{equation}
 where $|B'|=k$ and $l+k=m$. We define
 \begin{equation}\label{a}
    A=d^{-1}*(A'-\{a\}).
\end{equation}
A simple computation yields that
\begin{equation}\label{b}
 \begin{split}
 A=[0,l-1]\cup B,
 \end{split}
\end{equation}
with $B=\{a_1,a_2,...,a_k\}\subset\Z_p$ an ordered set, that is, for
$i,j\in[1,k]$, if $i\leq j$, then $a_i\leq a_j$. This process is
called the normalization of the set $A'$ and the set $A$ is called
the normal form of the set $A'$. The set $ A $ always be  written in
its normal form \eqref{b}. A important property of the normal form
is that
\begin{equation}\label{c}
    |2\hspace{.15cm}\widehat{}A\;|=|2\hspace{.15cm}\widehat{}A'\;|.
\end{equation}
The normal form of \eqref{ap} is the interval $[0,r-1]$. When we
refer to an arithmetic progression we will write this as an
interval. In what follows, we suppose that $2\;\widehat{}A \neq
\Z_p$. The restricted sumset of \eqref{b} is given by
\begin{equation}\label{d}
\begin{split}
2\hspace{.15cm}\widehat{}A=2\;\widehat{}\;[0,l-1]\cup2\;\,\widehat{}\,B\cup\left
([0,l-1]\widehat{+}B\right ).
\end{split}
\end{equation}
As well, we suppose that $|B|=k\geq 1$, since if $|B|=0$, then $A$
is an arithmetic progression.
\begin{remark}
\textit{Observe that $0,1\in A$ and if we suppose that
$2\;\widehat{}\;[0,l-1]\cup\left ([0,l-1]+B\right )$ is an interval
in \eqref{d}, then  the interval $2\;\widehat{}\;[0,l-1]\cup\left
([0,l-1]+B\right )$ must have one of the next forms:
$[1,a_k+l-1],\;$  $[a_1,p]\cup[1,2l-3]$ or
$[a_{s+1},p]\cup[1,a_s+l-1]$ with $a_1,a_k,a_s,a_{s+1}\in B$ and for
a fixed $s\in\{1,2,...,k\}$.}
\end{remark}
In the next theorem we suppose that $2\hspace{.15cm}\widehat{} A$ is
not an interval.
\begin{theorem}\label{T2}
Let $A \subseteq \Z_{p}$ as in \eqref{b} with $\left| A\right|
=m=l+k\geq 4$ and $A$ is not an arithmetic progression. Suppose that
$|2\hspace{.15cm}\widehat{} A\;|< p-1  $. If
$2\;\widehat{}\;[0,l-1]\cup \left( B+[0,l-1]\right) $ is an
interval, then
\begin{itemize}
\item [a)] If $2\;\widehat{}\;[0,l-1]\cup \left( B\widehat{+}[0,l-1]\right) $ is of the form $[1,a_k+l-1]$ and $a_{\min}=\min\{a\in B:l+k\leq a\}$, then
\begin{equation}\label{e}
|2\hspace{.15cm}\widehat{} A\;|\geq 2\left\vert A\right\vert-3
+w(l+k-2,a_{\min}).
\end{equation}
\item [b)] If $2\;\widehat{}\;[0,l-1]\cup \left( B\widehat{+}[0,l-1]\right) $ is of the form $[a_1,p]\cup[1,2l-3]$ and $a_{\max}=\max\{a\in B: a\leq p-k+1\}$, then
\begin{equation}\label{f}
|2\hspace{.15cm}\widehat{} A\;|\geq 2\left\vert A\right\vert-3
+w(a_{\max},p-k+1).
\end{equation}
\item [c)] If $2\;\widehat{}\;[0,l-1]\cup \left( B\widehat{+}[0,l-1]\right) $ is of the form $[a_{s+1},p-1]\cup[0,a_s+l-1]$ where $s\in [1,k-1]$ is fixed, $a_{\min}=\min\{a\in B:l+s\leq a\}$ and $a_{\max}=\max\{a\in B:a\leq p-k+s-1\}$, then
\begin{equation}\label{g}
|2\hspace{.15cm}\widehat{} A\;|\geq 2\left\vert A\right\vert
-3+w\left(l+s-2,a_{\min}\right)+w\left(a_{\max},p-k+s+1\right).
\end{equation}
\end{itemize}
\end{theorem}
\begin{proof}
\begin{itemize}
    \item [\textit{a)}] Let $A=[0,l-1]\cup B$. Since $2\;\widehat{}\;[0,l-1]\cup \left( B\widehat{+}[0,l-1]\right)=[1,a_{k}+l-1]$, we can rewrite \eqref{d} as
    \begin{equation}\label{h}
    2\hspace{.15cm}\widehat{} A\;=[1,a_{k}+l-1]\cup 2\;\widehat{}\;B.
    \end{equation}
    Since the set $A$ is not an arithmetic progression, we have that $ \{a\in B:l+k\leq a\}\neq \emptyset $, so there exists $ a_{\min}=\min\{a\in B:l+k\leq a\} $. Define $A_1=[0,l]\cup B_1,$ where $B_1=B\backslash\{a_k\}$. It is easy to see that $|A|=|A_1|$ and
    $$2\hspace{.15cm}\widehat{}A_1=2\;\widehat{}\;[0,l]\cup\left ([0,l]\widehat{+} B_1\right )\cup 2\;\widehat{\phantom{.}}B_1.$$
    Observe that $2\;\widehat{}\;[0,l]\cup\left ([0,l]\widehat{+} B_1\right )  $ remains an interval.
From the fact that $a_{k-1}+l\leq a_k+l-1$ we obtain that
    \begin{equation}\label{A_11}
    2\;\widehat{}\;[0,l]\cup \left ([0,l]\widehat{+} B_1\right )\subset 2\;\widehat{}\;[0,l-1]\cup \left([0,l-1]\widehat{+} B\right).
    \end{equation}
    On the other hand, since $B_1\subset B$ it follows that
    \begin{equation}\label{A_13}
    2\,\,\widehat{}\;B_1\subseteq 2\,\,\widehat{}\;B.
    \end{equation}
    From \eqref{A_11} and \eqref{A_13} we conclude
    \begin{equation}\label{ii}
    2\hspace{.15cm}\widehat{}A_1\subseteq 2\hspace{.15cm}\widehat{}A,
    \end{equation}
    and
    \begin{equation}\label{j}
    \left|2\hspace{.15cm}\widehat{}A_1\right|\leq\left|2\hspace{.15cm}\widehat{}A\right|.
    \end{equation}
    Now, if $ A_1 =[0,l+k-2]\cup \{a_{\min}\}$ or $ A_1=[0,l+q-1]\cup \lbrack l+q+1,a_{\min}] $  for some $ q\in[1,k-2] $, then the procedure stops. Otherwise, we recursively go on with the procedure until we obtain on of the following cases:
\begin{enumerate}
\item[i)] $A_j=[0,l+k-2]\cup \{a_{\min}\}$ for some $ j\in [2,k-1], $
\item[ii)] $A_j=[0,l+q-1]\cup \lbrack l+q+1,a_{\min}]$, for some $ j\in[2,k-2] $ and $q\in[1,k-2]$,
\item[iii)] $A_j=[0,l+k-3]\cup\{l+k-1,a_{\min}\} $ for some $ j\in[2,k-2] $.
\end{enumerate}

\begin{remark}\label{remark2}{
After applying the procedure, we get that $
2\hspace{.15cm}\widehat{}A_{j+1}\subseteq
2\hspace{.15cm}\widehat{}A_{j}$, with  $ j\in[1,k-3] $ in the cases
ii) and iii) and $ j\in[1,k-2] $ in the case i). Moreover, $
2\hspace{.15cm}\widehat{}A_{j}\subseteq  2\hspace{.15cm}\widehat{} A
$ (and consequently, $ |2\hspace{.15cm}\widehat{}A_{j}|\leq
|2\hspace{.15cm}\widehat{} A| $  ) for every possible $ j\in[1,k-2]
$. }
\end{remark}
Accordingly, we consider the following cases:
\begin{enumerate}
        \item If $A_j=[0,l+k-2]\cup \{a_{\min}\}$, for some $ j\in[1,k-1] $, then
        \begin{equation}\notag
        \begin{split}
        2\hspace{.15cm}\widehat{}A_j&=[1,a_{\min}+l+k-2]\\
        &=[1,2l+2k-3]\cup[2l+2k-2,a_{\min}+l+k-2]
        \end{split}
        \end{equation}
        Observe that $ a_{\min}=l+k-1+w(l+k-2,a_{\min}) $. Hence
$$\left\vert 2\hspace{.15cm}\widehat{} A_j\right\vert =2|A|-3+w(l+k-2,a_{\min}). $$
From this result and Remark \ref{r2} we conclude that
        \begin{equation}\notag
        \left\vert 2\hspace{.15cm}\widehat{} A\right\vert \geq2|A|-3+w(l+k-2,a_{\min}).
        \end{equation}
        \item If  $A_j=[0,l+q-1]\cup \lbrack l+q+1,a_{\min}]$, for some $ j\in[2,k-2] $ and $q\in[1,k-2]$, then
        \begin{equation}\notag
        2\hspace{.15cm}\widehat{} A_j=[1,2a_{\min}-1].
        \end{equation}
        It is not difficult to see that $ a_{\min}=l+k $. Then
        \begin{equation}
        \begin{split}
        \left\vert 2\hspace{.15cm}\widehat{} A_j\right\vert&=2(l+k)-1\\
        &\geq 2\left|A\right|-2\\
        &=2\left|A\right|-3+w(l+k-2,a_{\min})\\
        \end{split}
        \end{equation}
        From this result and  Remark \ref{remark2} we conclude
        \begin{equation}\notag
        \left\vert 2\hspace{.15cm}\widehat{} A\right\vert \geq 2|A|-3+w(l+k-2,a_{\min}).
        \end{equation}
        \item If $A_j=[0,l+k-3]\cup\{l+k-1,a_{\min}\} $ (with $ l+k\leq a_{\min} $) for some $ j\in[2,k-2] $, then $ a_{\min}=l+k-1+w(l+k-2,a_{\min})$ and
\begin{equation}
\begin{split}
2\hspace{.15cm}\widehat{} A_j&=[1,l+k-3+a_{\min}]\cup \{l+k-1+a_{\min}\}\\
&= [1,2(l+k)-4+w(l+k-2,a_{\min})] \cup \{l+k-1+a_{\min}\}.
\end{split}
\end{equation}
Now, observe that $  l+k-1+a_{\min}\notin
[1,2l+2k-4+w(l+k-2,a_{\min})] $ since if $p+1\leq l+k-1+a_{\min}$,
then $p-1\leq l+k-3+a_{\min}  $ which is impossible ($
|2\hspace{.15cm}\widehat{} A|< p-1 $). From the above and  Remark
\ref{remark2}, we conclude that
\begin{equation}\notag
\left\vert 2\hspace{.1cm}\widehat{} A\right\vert\geq \left\vert
2\hspace{.1cm}\widehat{} A_j\right\vert=2|A|-3+w(l+k-2,a_{\min}).
\end{equation}
\end{enumerate}
    \item [\textit{b)}] This case is a special case of $a)$. Observe that
    \begin{equation}\notag
    A'=-A+\{l-1\}=[0,l-1]\cup B',
    \end{equation}
    with $B'=\{p-a_k+l-1,p-a_{k-1}+l-1,...,p-a_1+l-1\}$. Therefore, the set $A'$ satisfy the case $a)$ and $w(a_{\max},p-k+1)=w(a_{\min},l+k)$ where $a_{\max}=\max\{a\in B: a\leq p-k-1\}$ and $a_{\min}=\min\{a\in B': l+k\leq a\}$. Hence
    \begin{equation}\notag
    \left\vert 2\hspace{.1cm}\widehat{} A\right\vert \geq 2|A|-3+w(a_{\max},p-k+1).
    \end{equation}
    \item[\textit{c)}] Let $ A=[0,l-1]\cup B $, where $B=B_1\cup B_2$ with $B_1=\{a_1,a_2,...,a_s\}$, $B_2=\{a_{s+1},a_{s+2},...,a_{k}\}$, and $ 2\leq s\leq k-1 $. Recall that $a_{\min}=\min\{a\in B_1:l+s\leq a\}$ and $a_{\max}=\max\{a\in B_2:a\leq p-k+s-1\}$. Define $ \Lambda^{(i)}=[0,l-1]\cup B_i $, $ (i=1,2) $. We apply the case $ a) $ to $ \Lambda^{(1)}$ and obtain $ \Lambda^{(1)}_{j_1}$ for some $ j_1 $. Let $A^{(1)}_{j_1} =\Lambda^{(1)}_{j_1}  $. We apply the case $ b) $ to $ \Lambda^{(2)}  $ and obtain $ \Lambda^{(2)}_{j_2}  $ for some $ j_2 $. Let $ A^{(2)}_{j_2} =-\Lambda^{(2)}_{j_2}+l-1 $. Finally, we set
\begin{equation}
A'=A^{(1)}_{j_1}\cup A^{(2)}_{j_2}=[c,d]\cup B_1'\cup B_2',
\end{equation}
where $ c\leq p-1 $, $ d\geq l+1,  $  $ B_1' $ is one of the
following sets $ \{a_{\min}\} $, $ [d+2,a_{\min}]$ or $
\{d+2,a_{\min}\} $, and $ B_2' $ is one of the following sets $
\{a_{\max}\} $, $  [a_{\max}, c-2] $ or $ \{a_{\max},c-2\} $.
Therefore, we should consider nine possible cases for $ A' $. Note
that $ 2\;\widehat{}\; [c,d]\cup ([c,d]+ (B_1'\cup B_2')) $ is still
an interval by construction and the Remark \ref{remark2} is still
valid.

    \begin{enumerate}
        \item[i)] If $A'=\{a_{\max}\}\cup[p-k+s+1,p-1]\cup[0,l+s-2]\cup\{a_{\min}\}$, then we define $A''=A'+\{k-s-1\}=\left\{a_{\max}+k-s-1\right\}\cup[0,l+k-3]\cup\left\{a_{\min}+k-s-1\right\}$. It is clear that
        \begin{equation}\label{r1}
        |2\hspace{.15cm}\widehat{}A'|=|2\hspace{.15cm}\widehat{}A''|,
        \end{equation}
 and
        \begin{equation}\notag
    \begin{split}
    2\hspace{.15cm}\widehat{}A''=[a_{\max}+k-s-1,p-1]&\cup[0,a_{\min}+2k+l-s-4]\cup\\
& \{a_{\min}+a_{\max}+2k-2s-2\}.
\end{split}
        \end{equation}
        In this case,
\begin{equation}
a_{\min}=l+s-1+w(l+s-2,a_{\min}),
\end{equation}
and
\begin{equation}
a_{\max}=p-k+s-w(a_{\max},p-k+s+1).
\end{equation} A simple computation yields
        \begin{equation}\label{p1}
        \begin{split}
        \left|2\hspace{.15cm}\widehat{}A''\right|   &\geq 2|A|-3+w(a_{\min},l+s-2)+w(a_{\max},p-k+s+1).
        \end{split}
        \end{equation}
        From \eqref{r1}, \eqref{p1} and Remark \ref{remark2}, we deduce that
        \begin{equation}\notag
        |2\hspace{.15cm}\widehat{}A|\geq 2|A|-3+w(l+s-2,a_{\min})+w(a_{\max},p-k+s+1).
        \end{equation}
        \item[ii)] If $A'=\{a_{\max}\}\cup[p-k+s+1,p-1]\cup[0,l+q-1]\cup[l+q+1,a_{\min}]$ for some $q\in[1,s-2]$, then we define \begin{equation}
\begin{split}
A''&=A'+\left\{k-s-1\right\}\\
&=\left\{a_{\max}+k-s-1\right\}\cup [0,l+i+k-s-2]\cup\\
&\phantom{fffffffffffffffffffffff} [l+i+k-s,a_{\min}+k-s-1].
\end{split}
\end{equation} We recall that
        \begin{equation}\label{r2}
        |2\hspace{.15cm}\widehat{}A'|=|2\hspace{.15cm}\widehat{}A''|.
        \end{equation}
        In this case,
\begin{equation}
a_{\min}=l+s,
\end{equation}
so $w(l+s-2,a_{\min})=1$, and
\begin{equation}
 a_{\max}=p-k+s-w(a_{\max},p-k+s+1).
\end{equation}  Hence
        \begin{equation}\notag
        2\hspace{.15cm}\widehat{}A''=[a_{\max}+k-s-1,p-1]\cup[0,2l+2k-3],
        \end{equation}
        and
        \begin{equation}\label{p2}
        \begin{split}
        |2\hspace{.15cm}\widehat{}A''|&=p-(a_{\max}+k-s-1)+2l+2k-2  \\
        &=p-k+s-a_{\max}+2l+2k-1\\
        &\geq 2|A|-3+w(l+s-2,a_{\min})+w(a_{\max},p-k+s+1).
        \end{split}
        \end{equation}
        From \eqref{r2}, \eqref{p2} and Remark \ref{remark2}, we get that
        \begin{equation}\notag
        |2\hspace{.15cm}\widehat{}A|\geq 2|A|-3+w(a_{\max},p-k+s+1)+w(l+s-2,a_{\min}).
        \end{equation}
        \item[iii)] If $ A'=[a_{\max},p-q-2]\cup [p-q,p-1]\cup[0,l+s-2]\cup\{a_{\min}\} $ for some $ q\in[1,k-s-2] $, then we apply Case ii) to the set $ -A'+\{l+s-2\} $ and we get the result.
        \item[iv)] If $A'=\{a_{\max}\}\cup[p-k+s+1,p-1]\cup[0,l+s-3]\cup\{l+s-1,a_{\min}\}$, then we define $ A''=A'+\{k-s-1\}=\{a_{\max}+k-s-1\}\cup[0,k+l-4]\cup\{l+k-2,a_{\min}+k-s-1\} $. Observe that
\begin{equation}
a_{\max}=p-k+s-w(a_{\max},p-k+s+1),
\end{equation}
and
\begin{equation}
a_{\min}=l+s-1+w(l+s-2,a_{\min}).
\end{equation}
Hence
\begin{equation}
\begin{split}
2\hspace{.15cm}\widehat{}A''&=[a_{\max}+k-s-1,p]\cup [1,a_{\min}+2k+l-s-5]\cup \\
&\{a_{\max}+2k-s-3,a_{\min}+2k+l-s-3,a_{\max}+a_{\min}+2k-2s-2\}.
\end{split}
\end{equation}
By hypothesis, $ |2\hspace{.15cm}\widehat{}A|\leq p-2 $ and so
$$ a_{\min}+2k+l-s-5\leq a_{\max}+k-s-4, $$
and
$$ a_{\min}+2k+l-s-3< a_{\max}+k-s-1. $$

Therefore
\begin{equation}
\begin{split}
|2\hspace{.15cm}\widehat{}A|\geq |2\hspace{.15cm}\widehat{}A''|&\geq\left |[a_{\max}+k-s-1,p-1]\right |+\left |[0,a_{\min}+2k+l-s-5]\right |+1\\
&=2|A|-3+w(l+s-2,a_{\min})+w(a_{\max},p-k+s+1).
\end{split}
\end{equation}
\item[v)] If $A'=\{a_{\max},p-k+s\}\cup[p-k+s+2,p-1]\cup[0,l+s-2]\cup\{a_{\min}\}$, then we apply Case iv) to the set $ -A'+\{l+s-2\} $ and we obtain the result.
        \item [vi)] If $A'=[a_{\max},p-q_2-2]\cup[p-q_2,p-1]\cup[0,l+q_1-1]\cup[l+q_1+1,a_{\min}]$ for some $ q_1\in[1,s-2] $ and $ q_2\in[1,k-s-2] $, then we define
\begin{equation}
\begin{split}
A''&=A'+\left\{p-a_{\max}\right\}\\
&=[0,p-a_{\max}-j-2]\cup [p-a_{\max}-j,p-a_{\max}]\cup\\
& \phantom{ffffffffffffff}[p-a_{\max}+1,a_j+p+l+i-a_{\max}-1]\cup\\
&
\phantom{ffffffffffffffffff}[p-a_{\max}+l+i+1,p-a_{\max}+a_{\min}].
\end{split}
\end{equation}
In this case, $a_{\min}=l+s$ and $a_{\max}=p-k+s-1$, so
$w(l+s-2,a_{\min})=1$ and $w(a_{\max},p-k+s+1)=1$. On the other
hand, we have
        \begin{equation}\label{r4}
        |2\hspace{.15cm}\widehat{}A'|=|2\hspace{.15cm}\widehat{}A''|.
        \end{equation}
        Then
        \begin{equation}\notag
        2\hspace{.15cm}\widehat{}A''=[1,2p-2a_{\max}+2a_{\min}-1].
        \end{equation}
        A simple computation yields
        \begin{equation}\label{p4}
        \begin{split}
        |2\hspace{.15cm}\widehat{}A''|&=2p-2(p-k+s-1)+2(l+s)-1\\
        &=2|A|-3+4\\
        &\geq 2|A|-3+w(l+s-2,a_{\min})+w(a_{\max},p-k+s+1).
        \end{split}
        \end{equation}
        From \eqref{r4}, \eqref{p4} and Remark \ref{remark2}, we deduce that
        $$|2\hspace{.15cm}\widehat{}A|\geq 2|A|-3+w(l+s-2,a_{\min})+w(a_{\max},p-k+s+1).$$
    \item [vii)] If $A'=[a_{\max},p-q-2]\cup[p-q,p-1]\cup[0,l+s-3]\cup\{l+s-1,a_{\min}\}$ for some $ q\in[1,k-s-2] $, then we define
\begin{equation}
\begin{split}
 A''=A'+\{p-a_{\max}\}=[0,&2p-q-a_{\max}-2]\cup\\
&[2p-q-a_{\max},2p-1-a_{\max}]\cup\\
 &[p-a_{\max},p+l+s-3-a_{\max}]\cup\\
&\{p+l+s-1-a_{\max},p+a_{\min}-a_{\max}\}.
\end{split}
\end{equation}
In this case,
\begin{equation}
a_{\max}=p-k+s-1
\end{equation}
and
\begin{equation}
a_{\min}=l+s-1+w(l+s-2,a_{\min}).
\end{equation}
 On the other hand, it is not difficult to see that
\begin{equation}
2\hspace{.15cm}\widehat{}A''=
[1,2p+a_{\min}-2a_{\max}+l+s-3]\cup\{2p+l+s-2a_{\max}+a_{\min}-1\},
\end{equation}
and
\begin{equation}\label{r5}
|2\hspace{.15cm}\widehat{}A'|\geq|2\;\widehat\;A''|.
\end{equation}
It follows that
\begin{equation}\label{p5}
        \begin{split}
        |2\hspace{.15cm}\widehat{}A''|&\geq |[1,2p+a_{\min}-2a_{\max}+l+s-3]|\\
        &= 2p-2(p-k+s-1)+\\
&\phantom{ccccccc}(l+s-1+w(l+s-2,a_{\min}))+l+s-3\\
        &=2|A|-3+w(l+s-2,a_{\min})+1\\
        &\geq 2|A|-3+w(l+s-2,a_{\min})+w(a_{\max},p-k+s+1).\\
        \end{split}
        \end{equation}
    The last inequality is true because $ w(a_{\max},p-k+s+1)=1 $. From \eqref{r5}, \eqref{p5} and Remark \ref{remark2}, we get that
        $$|2\hspace{.15cm}\widehat{}A|\geq 2|A|-3+w(l+s-2,a_{\min})+w(a_{\max},p-k+s+1).$$
\item[viii)]If $ A'=\{a_{\max},p-k+s\}\cup[p-k+2+2,p-1]\cup[0,l+q-1]\cup[l+q+1,a_{\min}] $ for some $ q\in[1,s-2] $, then we apply Case vii) to the set $ -A'+\{l+q-1\} $ and we get the result.
\item [ix)] If $A'=\{a_{\max},p-k+s\}\cup[p-k+s+2,p-1]\cup[0,l+s-3]\cup\{l+s-1,a_{\min}\}$, then
 \begin{equation}\label{min}
a_{\min}=l+s-1+w(l+s-2,a_{\min})
\end{equation}
and
\begin{equation}\label{max}
a_{\max}= p-k+s-w(a_{\max},p-k+s+1).
\end{equation}Define
\begin{equation}
\begin{split}
 A''&=A'+\{k-s-2\}\\
&=\{a_{\max}+k-s-2,p-2\}\cup[0,l+k-5] \cup\\
&\phantom{ffffffffffffffffffffff}\{l+k-3,a_{\min}+k-s-2\}.
\end{split}
\end{equation}
As in the above cases, we have that
\begin{equation}\label{r6}
|2\hspace{.15cm}\widehat{}A'|\geq|2\;\widehat\;A''|.
\end{equation}
Let $ B_1=\{a_{\max}+k-s-2,p-2\} $ and $
B_2=\{l+k-3,a_{\min}+k-s-2\} $. Accordingly
\begin{equation}
\begin{split}
2\hspace{.15cm}\widehat{}A''=[a_{\max}+k-s-2,p-1]\cup[0,a_{\min}&+2k+l-s-7]\cup\\
&2\hspace{.15cm}\widehat{}\hspace{.1cm}B_1\cup(B_1+B_2)\cup
2\hspace{.15cm}\widehat{}\hspace{.1cm}B_2
\end{split}
\end{equation}
An important property of the above equality is that
\begin{equation}\label{ww}
|2\hspace{.15cm}\widehat{}A''|\geq 2|A|-5+w(l+s-2,a_{\min})
+w(a_{\max},p-k+s+1).
\end{equation}
Observe that
\begin{equation}
a_{\min}+2k+l-s-7< a_{\max}+2k+l-s-5,
\end{equation}
and
\begin{equation}
a_{\min}+2k+l-s-5< a_{\max}+k-s-2.
\end{equation}
It follows that
$$ \; l+k-3+a_{\max}+k-s-2\notin [a_{\max}+k-s-2,p-1]\cup[0,a_{\min}+2k+l-s-7],$$
and
\begin{equation}\label{key}
 \;l+k-3+a_{\min}+k-s-2\notin [a_{\max}+k-s-2,p-1]\cup[0,a_{\min}+2k+l-s-7].
\end{equation}
\hspace{-1.8cm}The above elements are different because $ a_{\min}<
a_{ \max} $.  Therefore
\begin{equation}\label{p6bis}
\begin{split}
|2\hspace{.15cm}\widehat{}A''|\geq & 2|A|-5+w(l+s-2,a_{\min})+w(a_{\max},p-k+s+1)+2\\
&=2|A|-3+w(l+s-2,a_{\min})+w(a_{\max},p-k+s+1).
\end{split}
\end{equation}
\end{enumerate}
\end{itemize}
\end{proof}
\begin{remark}\label{remark1}
\textit{Suppose that the conditions of the previous theorem are
valid and $w(l+k-2,a_{\min})=2$ in the part $a)$. Then, we must have
that  $a_{\min}=l+k+1$. Therefore $A$ cannot be of the form
$[0,l-1]\cup[l+1,l+k]$. Otherwise, if $A=[0,l-1]\cup[l+1,l+k]$ with
$k=1$, hence $A=[0,l-1]\cup \{l+1\}$ which implies that
$|2\hspace{.15cm}\widehat{}A|=2|A|-2$. For $k\geq 2$, since
$2\hspace{.15cm}\widehat{}\hspace{.15cm}[0,l-1]\cup \left
([0,l-1]+[l+1,l+k]\right ) $ is an interval, we have that}
\begin{equation}\label{}
2\hspace{.15cm}\widehat{}A\supseteq[1,2l+2k-1].
\end{equation}
\textit{Therefore, $|2\hspace{.15cm}\widehat{}A|\geq 2|A|-1$. A
similar argument can be used for part $b)$ of the previous theorem.
For part $c)$,  since $w(l+s-2,a_{\min})\geq 1$ and
$w(a_{\max},p-k+s+1)\geq 1$, we have that $A$ cannot be of the form
$[0,l-1]\cup[l+1,l+k]$ for $k\geq 1$. We conclude that under the
conditions of the previous theorem it is always true that
$|2\hspace{.15cm}\widehat{}A|\geq 2|A|-1$.}
\end{remark}
Suppose that $2\hspace{.15cm}\widehat{} A$ is an interval. Since
$0,1\in A$, we can suppose that it is of the form $[d,p-1]\cup[0,c]$
with $c,d\in\Z_p$. We partition $B$ into two sets $B_1$ and $B_2$ as
follows:
\begin{equation}\label{B_1}
    B_1=\left\{a\in B:(a\;\widehat {+}\;[0,l-1])\subset [0,c]\right\},
\end{equation}
and
\begin{equation}\label{B_2}
B_2=\left\{a\in B:(a\;\widehat {+}\;[0,l-1])\subset
[d,p-1]\cup[0,2l-3]\right\},
\end{equation}
with $|B_1|=s$ and $|B_2|=k-s$.

\begin{theorem}\label{t2}
Let $A \subseteq \Z_{p}$ such as in \eqref{b} with $\left| A\right|
=m=l+k\geq 4$ and $A$ is not an arithmetic progression . If
$2\hspace{0.15cm}\widehat{} A $ is an interval with $
|2\hspace{0.15cm}\widehat{} A|<p-1 $, then
\begin{itemize}
\item [a)] If $2\hspace{0.15cm}\widehat{} A $ is of the form $[1,c]$ and $a_{\min}=\min\{a\in B:l+k\leq a\}$, then
\begin{equation}\label{e'}
|2\hspace{0.15cm}\widehat{} A|\geq 2\left\vert A\right\vert-3
+w(l+k-2,a_{\min}).
\end{equation}
\item [b)] If $2\hspace{0.15cm}\widehat{} A $ is of the form $[d,p]\cup[1,2l-3]$ and $a_{\max}=\max\{a\in B: a\leq p-k-1\}$, then
\begin{equation}\label{f'}
|2\hspace{0.15cm}\widehat{} A|\geq 2\left\vert A\right\vert-3
+w(a_{\max},p-k+2).
\end{equation}
\item [c)] If $2\hspace{0.15cm}\widehat{}A$ is of the form $[d,p]\cup[1,c]$, $a_{\min}=\min\{a\in B_1:l+s\leq a\}$ and $a_{\max}=\max\{a\in B_2:a\leq p-k+s-1\}$ with $B_1,B_2$ as in \eqref{B_1},\eqref{B_2}; respectively, then
\begin{equation}\label{g'}
|2\hspace{.15cm}\widehat{} A|\geq 2\left\vert A\right\vert
-3+w\left(l+s-2,a_{\min}\right)+w\left(a_{\max},p-k+s+1\right).
\end{equation}
\end{itemize}
\end{theorem}
\begin{proof}
Observe that if  $2\hspace{.15cm}\widehat{} A$ is an interval, then it does not imply that $2\hspace{.15cm}\widehat{}\; [0,l-1]\cup \left ([0,l-1]\widehat{+} B\right )$ is an interval. However, $2\hspace{.15cm}\widehat{}\;[0,l-1]\cup \left ([0,l-1]\widehat{+} B\right )\subset 2\hspace{.15cm}\widehat{} A$ (properly contained). In cases $ a) $ and $ b) $ of the Theorem, the results are immediate consequence of Theorem \ref{T2}. For case $ c) $, apply the recursive procedure of Theorem \ref{T2} to $ A $, $B_1$ and $B_2$ as defined in \eqref{B_1} and \eqref{B_2}.\\
\end{proof}
\begin{remark}
\textit{Suppose that the conditions of Theorem \ref{t2} are valid.
If we apply the same arguments of Remark \ref{remark1} to Theorem
\ref{t2}, then we can conclude that  it is always true that
$|2\hspace{.15cm}\widehat{}A|\geq 2|A|-1$.}
\end{remark}
Let $A\subseteq\Z_p$ not an arithmetic progression. Define
$$A=\bigcup_{m=1}^{r}I_m,\qquad r\geq 2$$
and $I_m$ are disjoint intervals ($m=1,\dots,r$). Let
$H:A\rightarrow \N_0$ defined by $H(A)=r-1$.

Now suppose that $2\hspace{0.15cm}\widehat{}\;[0,l-1]\cup
2\hspace{0.15cm}\widehat{}\;B$ is an interval and
$2\hspace{0.15cm}\widehat{}A$ is not. We partition $B$ into three
sets as follows:
\begin{equation}\label{r}
  \begin{split}
  B=\{a_1,a_2,...,a_{n-1}\}\cup\{a_n,a_{n+1},...,a_N\}\cup\{a_{N+1},...,a_{k}\}=B_1\cup B_2\cup B_3,
  \end{split}
  \end{equation}
where
\begin{equation}\label{s}
    a_n=\min\left\{a\in B:(a\widehat{+}[0,l-1])\cap(2\hspace{0.15cm}\widehat{}\;[0,l-1]\cup 2\hspace{0.15cm}\widehat{}\,B)=\emptyset\right\},
\end{equation}
and
\begin{equation}\label{t}
    a_N=\max\left\{a\in B:(a\widehat{+}[0,l-1])\cap(2\hspace{0.15cm}\widehat{}\;[0,l-1]\cup 2\hspace{0.15cm}\widehat{}\, B)=\emptyset\right\}.
\end{equation}
The elements $a_n  $ and $ a_N $ are well defined elements of $ B $
since the set $2\hspace{0.15cm}\widehat{}A$ is not an interval.
\begin{theorem}\label{t3}
Let $A=[0,l-1]\cup B \subseteq \Z_{p}$ with $\left| A\right|
=m=l+k\geq 4$. Suppose that $A$ is not an arithmetic progression. If
$2\hspace{0.15cm}\widehat{}\;[0,l-1]\cup
2\hspace{0.15cm}\widehat{}\,B$ is an interval and
$2\hspace{0.15cm}\widehat{} A$ is not an interval, then
\begin{equation}\label{u}
\left\vert 2\;\widehat{} A\right\vert \geq 2\left\vert A\right\vert
-2+H(B_2).
\end{equation}
\end{theorem}
\begin{proof}
We define $B=B_1\cup B_2\cup B_3$ as in \eqref{r} and $a_n,a_N$ as
in \eqref{s} and \eqref{t}, respectively. Consider
\begin{equation}\label{v}
\begin{split}
2\hspace{0.15cm}\widehat{} A=2\hspace{0.15cm}\widehat{}\;[0,l-1]\cup
2\hspace{0.15cm}\widehat{}\, B\cup (B_1\widehat{+}[0,l-1])\cup
&(B_2\widehat{+}[0,l-1])\cup (B_3\widehat{+}[0,l-1]).
\end{split}
\end{equation}
It is easy to see that $B_2\widehat{+} B$ is a subset of the
interval $2\;\widehat{}\;[0,l-1]\cup 2\;\widehat{}\;B$, then
\begin{equation}\label{v1}
\begin{split}
B_2\widehat{+} B\subseteq[a_n,a_N]\;\widehat{+}\, B\subseteq
2\hspace{0.15cm}\widehat{}\;[0,l-1]\cup
2\hspace{0.15cm}\widehat{}\,B.
\end{split}
\end{equation}
We define $B'_2=[a_n,a_n+N-n]$. Observe that $|B_2'|=|B_2|=N-n+1$.
From \eqref{v1} and the fact that $a_n+N-n\leq a_{N}$, we obtain
\begin{equation}\label{w}
\begin{split}
B'_2\;\widehat{+}\, B\subseteq 2\;\widehat{}\;[0,l-1]\cup
2\;\widehat{}\;B.
\end{split}
\end{equation}
An important fact is that $2\;\widehat{}\;B'_2\subseteq
2\;\widehat{}\;[0,l-1]\cup 2\;\widehat{}\;B$. We define
\begin{equation}\notag
\begin{split}
A'=[0,l-1]\cup B'
\end{split}
\end{equation}
with $B'= B_1\cup B'_2\cup B_3.$ Observe that $2\;\widehat{}\;B'\subseteq2\;\widehat{}\;[0,l-1]\cup2\;\widehat{}\;B$. 
Hence
\begin{equation}\label{x}
\begin{split}
    2\hspace{.1cm}\widehat{} A'&=2\;\widehat{}\;[0,l-1]\cup 2\;\widehat{}\,B' \cup (B_1\widehat{+}[0,l-1])\cup (B_2'\widehat{+}[0,l-1])\cup (B_3\widehat{+}[0,l-1])\\
&\subseteq 2\;\widehat{}\;[0,l-1]\cup 2\;\widehat{}\,B \cup
(B_1\widehat{+}[0,l-1])\cup (B_2'\widehat{+}[0,l-1])\cup
(B_3\widehat{+}[0,l-1])
\end{split}
\end{equation}
Using \eqref{v},\eqref{w}, and \eqref{x}, we obtain that
\begin{equation}\notag
\begin{split}
2\; \widehat {} A'\; \backslash \left ( B'_2\widehat{+}[0,l-1]\right
)\subseteq 2\; \widehat {} A\;\backslash \left (
B_2\widehat{+}[0,l-1]\right )
\end{split}
\end{equation}
and
\begin{equation}\label{y}
\left|2\; \widehat {} A'\; \backslash \left (
B'_2\widehat{+}[0,l-1]\right )\right|\leq \left|2\; \widehat {}
A\;\backslash \left ( B_2\widehat{+}[0,l-1]\right )\right|.
\end{equation}
Since $B_2\widehat{+}[0,l-1]$ and $B'_2\widehat{+}[0,l-1]$ are
subset of $2\; \widehat {} A$ and $2\; \widehat {} A'$;
respectively, then
\begin{equation}\label{z}
\begin{split}
|2\; \widehat {} A'\; \backslash  \left (
B'_2\widehat{+}[0,l-1]\right )|= \left|2\; \widehat {}
A'\;\right|-\left| \left ( B'_2\widehat{+}[0,l-1]\right )\right|,
\end{split}
\end{equation}
and
\begin{equation}\label{a1}
\begin{split}
\left|2\; \widehat {} A\;\backslash \left (
B_2\widehat{+}[0,l-1]\right )\right|=\left|2\; \widehat {}
A\;\right|-\left| \left ( B_2\widehat{+}[0,l-1]\right )\right|.
\end{split}
\end{equation}
From \eqref{y}, \eqref{z} and \eqref{a1} it follows that
\begin{equation}\label{b1}
\left|2\; \widehat {}
A'\;\right|+\left(\left|B_2+[0,l-1]\right|-\left|B'_2+[0,l-1]\right|\right)\leq
\left|2\; \widehat {} A\;\right|.
\end{equation}
By hypothesis $A'$ is not an arithmetic progression, then
$|2\;\widehat{} A'|\geq 2|A'|-2$, see \cite{Karolyi}. Therefore
\begin{equation}\label{c1}
|2\;\widehat{} A'|\geq 2|A|-2,
\end{equation}
since $|A|=|A'|$. On the other hand, it is not difficult to see that
\begin{equation}\label{d1}
    \left|B_2+[0,l-1]\right|-\left|B'_2+[0,l-1]\right|\geq H(B_2).
\end{equation}
From \eqref{b1},\eqref{c1} and \eqref{d1} we conclude that
\begin{equation}\notag\label{e1}
\left|2\; \widehat {} A\;\right|\geq    2|A|-2+ H(B_2),
\end{equation}
\end{proof}
\begin{remark}
\textit{Under the conditions of Theorem \ref{t3}, it is easy to
verify that $A$ can never be of the form $[0,l-1]\cup[l+1,l+k]$. If
$H(B_2)=0$, then $B_2$ is an interval of the form $[a_n,a_n+N-n]$
with $|B_2|=N-n+1\leq l$. Suppose that $B_1=B_3=\emptyset$. We will
consider  some cases related with the structure of $B_2$. }
\begin{itemize}
\item \textit{If $B_2=\left [\frac{p+1}{2},\frac{p+1}{2}+\left \lfloor\frac{ l+1}{2}\right \rfloor+i\right ],$ $i=0,1,2,...$, then $[0,l-1]$ would not be the longest arithmetic progression.  The longest arithmetic progression would be the one that begins at $a=0$ and have difference $d=\frac{p+1}{2}$.}
\item \textit{If $B_2=\left [\frac{p+1}{2},\frac{p+1}{2}+\left \lfloor\frac{ l+1}{2}\right \rfloor-1\right ],$ then $|B_2|=\left \lfloor\frac{ l+1}{2}\right \rfloor$ and}
$$2\; \widehat {} A=[1,2l-3]\cup \left [\frac{p+1}{2},\frac{p+1}{2}+\left \lfloor\frac{ l+1}{2}\right \rfloor+l-2\right ].$$
\textit{On the other hand, it is true that}
$$\left |\left [\frac{p+1}{2},\frac{p+1}{2}+\left \lfloor\frac{ l+1}{2}\right \rfloor+l-2\right ]\right |=\left \lfloor\frac{ l+1}{2}\right  \rfloor+l-1.$$
\textit{It is not difficult to see that }
$$2\left \lfloor\frac{ l+1}{2}\right  \rfloor+\left \lfloor\frac{l}{2}\right\rfloor-1= \left \lfloor\frac{ l+1}{2}\right  \rfloor+l-1.$$

\textit{It follows that}
$$|2\; \widehat {} A|=2|A|-3+\left \lfloor \dfrac{l}{2}\right \rfloor-1.$$
\textit{Observe that if $l=2,3,4,5$, then the hypothesis of theorem
are not valid. Even so, suppose that $l=4$. We have that
$A=[0,3]\cup [12,13]$ and this set can be transformed to
$A'=2*A=[0,4]\cup \{6\}$ and so $|2\; \widehat {} A|=|2\; \widehat
{} A'|=2|A'|-2=10$. The case $l=5$ can be handled in a similar way.
If $l\geq 6$, then we get that $|2\; \widehat {} A|\geq 2|A|-1$.}
\end{itemize}
\end{remark}

\subsection*{Acknowledgements}
We thank Prof. Luis Verde--Star for the motivation during the writing of paper. 


\end{document}